\documentclass{article}

\usepackage{arxiv}

\usepackage[utf8]{inputenc} 
\usepackage[T1]{fontenc}    
\usepackage{hyperref}       
\usepackage{url}            
\usepackage{booktabs}       
\usepackage{amsfonts}       
\usepackage{nicefrac}       
\usepackage{microtype}      
\usepackage{lipsum}
\usepackage{graphicx}
\graphicspath{ {./images/} }

\usepackage[numbers,sort&compress]{natbib}

\usepackage[utf8]{inputenc}
\usepackage[T1]{fontenc}
\usepackage{graphicx}          

\usepackage{graphicx}
\usepackage{dcolumn}

\usepackage{amsmath}
\usepackage{amsthm}
\usepackage{dsfont}
\usepackage{color}
\usepackage{comment}
\usepackage{marginnote}
\usepackage{tabularx}
\usepackage{ragged2e}
\usepackage{placeins}

\newtheorem{theorem}{Theorem}

\newtheorem{lemma}{Lemma}

\title{Node-Wise Dynamic Optimal Control for Evolutionary Games on General Multilayer Networks}

\author{
 Rio Aurachman \\
  School of Industrial Engineering\\
  Telkom University, Indonesia\\
  School of Electrical and Electronic Engineering\\
  University of Sheffield\\
  Sheffield, United Kingdom \\\texttt{rioaurachman@telkomuniversity.ac.id, raurachman1@sheffield.ac.uk} \\
   \And
 Giuliano Punzo \\
  School of Electrical and Electronic Engineering\\
  University of Sheffield\\
  Sheffield, United Kingdom \\
  \texttt{g.punzo@sheffield.ac.uk} \\
}

\begin{document}
\maketitle
\begin{abstract}
Promoting cooperative behaviour amongst decision makers has key implications for the long term sustainability of social systems. Incentives can promote cooperation in situations where defection is more favourable. Previous research has identified optimal decentralised incentives in structured populations where the social environment is described by a generic single-layer networks. Optimality is here intended in the sense of cost minimisation. Here, we provide and optimal incentive strategy for structured populations in general network, a rather unexplored case.  Further, we look at populations where, through the network structure, each player interact with one set of neighbours while contributing to an opinion diffusion dynamics from a second set of neighbours. Hence we cover the case of multilayer networks. To fill this gap, we provide a solution to the optimal control problem by solving the Hamilton-Jacobi-Bellman equation and derive an analytic solution for distributing incentives in multilayer networks. By implementing the dynamics of the networked Prisoner's Dilemma Game, we provide a feedback control loop that yields an optimal incentive distribution over time. The dynamic incentive depends on the current state of cooperation, the game-layer network, the strategy-diffusion layer, and the payoff matrix design. We found that the optimal incentive is node-wise, unique to each node on the network, influenced by its relative position according to the adjacency matrices of both layers. We also found that the optimal solution is not exclusively reward or punishment. While one node can receive a reward, the other node may receive punishment at the same time, depending on its current level of cooperation, relative to other nodes. We provide analytic solutions and numerical validations for the cases studied, comparing our results to the existing literature. 
\end{abstract}


\section{Introduction}
In socio-technical systems, the performance of the technical subsystem depends on the behaviour of the social subsystem \cite{Ropohl1999}. One important aspect that has been researched and explored in human societies is cooperative behaviour, which has been declared one of the grand scientific challenges of the 21st century \cite{Perc2017}. The dynamics of the cooperation can be modelled using an evolutionary game on a network \cite{Ohtsuki2006b} \cite{Zhu2021}. Besides the Stackelberg game and mechanism design \cite{zhu2025revisiting}, public goods games are a type of game that can be used to understand the complexity of cooperative behaviour within sociotechnical systems, for example, network security \cite{Das2022}.

Public goods arise from the cumulative contributions of individuals and are shared by all members of the group, with the understanding that the benefit cannot be achieved individually \cite{Du2022}. Since contributing incurs an individual cost while the benefits are non-excludable, individuals may be incentivised to free-ride by withholding contributions while still benefiting from the public good. As free-riding spreads through strategic adaptation, cooperation progressively declines, potentially leading to the collapse of the public good \cite{nunn1978public}.

The evolutionary Prisoner's Dilemma on networks provides a suitable mathematical framework for analysing the emergence of social dilemmas, similar to those observed in public goods games. Defection is the individually rational strategy regardless of the opponent’s action. As rational agents adopt the same strategy, the population may converge to mutual defection, resulting in a suboptimal outcome for all despite mutual cooperation yielding higher collective payoffs.

One approach to mitigate the tendency toward widespread defection and promote cooperation is to introduce incentives, such as rewards for cooperators, punishment for defectors \cite{Szolnoki2010}, or exclusion for free-riders \cite{Liu2018}. Incentives are usually introduced by the regulator (a third party) to promote cooperation \cite{zu2022reward}, which is always costly and may thus fail to be sustained \cite{zhou2022costly}. It is necessary to formulate an effective, optimal incentive strategy that reflects the player's performance to minimise total cost \cite{Han2018}. One way to design the incentive is to formulate an optimal control problem for decentralised reward or punishment in a networked public good game \cite{wang2022decentralized} and a Prisoner's Dilemma \cite{wang2023optimization}.The optimal strategy reflect to the performance of the player. In a Prisoner's Dilemma, an incentive is implemented as an additional payoff on top of the basic payoff, serving as a reward for the player who chooses to cooperate \cite{wang2025optimally}. Another way to control cooperation throughout the network is to introduce cooperative zealots, thereby giving surrounding agents an incentive to change their behaviour to also cooperate \cite{Riehl2017}. One study is a simulation in which the reward is provided by another player as a rewarding cooperator (RC) \cite{Szolnoki2010}.

The network structures considered in previous studies vary. Many works have focused on single layer regular lattice networks \cite{wang2022decentralized}\cite{wang2023optimization}\cite{wang2025optimally}\cite{Szolnoki2010}, while others have considered complete networks \cite{wang2025optimally}\cite{Riehl2017} and star networks \cite{Riehl2017}. However, only a limited number of recent studies have formulated optimal control strategies for general networks \cite{Riehl2017}, by relying primarily on algorithmic rather than analytical solutions. In many real-world systems, interaction networks do not conform to generic topologies such as regular lattices, complete graphs, Erdős–Rényi graphs, or other canonical random network models. Instead, they often exhibit application-specific structures with unique characteristics. Besides the design of optimal incentive mechanisms for general networks remains an open problem, the corresponding problem when the same player is influenced by different contacts or neighbours is also still unresolved. This would be immediate to model through a multilayer setting where  game interactions and strategy dynamics are mediated by distinct graph structures. 

As for modelling the strategy set, many recent studies adopt a pure strategy framework, in which players choose either full defection (0) or full cooperation (1) \cite{wang2025optimally,Szolnoki2010, Riehl2017}. However, cooperation may vary continuously between these two extreme values. Under a continuous strategy framework, the design of reward and punishment mechanisms becomes substantially more complex, as incentives may depend not only on an individual’s current strategy but also on the player’s network position, the strategies and cooperation levels of neighbouring players, and the incentive distributions across the network. Addressing this challenge requires a new modelling framework.

This work benefits and extends recent studies by addressing several key limitations. Using the Hamilton-Jacobi-Bellman framework \cite{wang2022decentralized}, we develop an analytical approach to determine the optimal control of an evolutionary Prisoner's Dilemma with continuous strategies on a general multilayer network. Optimal control theory enables the design of a dynamic incentive mechanism that minimises cumulative cost over time. Consistent with previous studies, incentives are determined adaptively based on the current strategy profile. However, the continuous strategy framework allows for a more refined and selective allocation of incentives than binary models. 

The general network structure yields node-specific incentive dynamics, while the multilayer setting provides insight into how a player’s structural position across layers influences incentive allocation. A key finding is that, depending on the system state and relative cooperation levels among neighbouring players, a node may receive either rewards or penalties along the system trajectory.

\section{Model and Framework}
\subsection{Network}
In this study, cooperation is treated as a dynamic variable that varies uniquely among nodes within a network. We introduce a multilayer network as a pair
\begin{equation}\label{genericnetwork}
\mathcal{M}=(\mathcal{G},\mathcal{E}),
\end{equation}
where $\mathcal{G}=\{\mathbb{G}\}$ is a tuple of graphs $\mathcal{G}=(G^{(1)},G^{(2)},...G^{(g)})$ and $\mathcal{E}$ is a tuple of edge sets connecting nodes between elements $\mathcal{G}$, that is $\mathcal{E}=(\mathcal{E}^{12},\mathcal{E}^{23},\mathcal{E}^{13},...\mathcal{E}^{C(g,2)})$. 

We shall assume that each node is connected to at least another node in every layer, and every graph on each layer has its own vertices and edges. Here, $G=(V,E)$ represents a graph where $V=\{1,2,\ldots,N\}$ denotes the set of nodes and $E \subseteq V \times V$ represents the set of ordered node pairs (edges).

The relationships between nodes on each layer are represented using an adjacency matrix. This $N \times N$ matrix, where $N$ is the population size, indicates whether interactions exist between pairs of nodes. Specifically, the layer $g$ graph's adjacency matrix $A^g=a^g_{ij}$ is defined such that $a^g_{ij}=1$ if there is an edge between nodes $i$ and $j$ in layer $g$ (i.e., $(i,j) \in E^g$), and $a_{ij}=0$ otherwise.  

We define the network in this study as a two-layer undirected network. So the $\mathcal{E}=\mathcal{E}_{12}$ is a matrix with size $N^1 \times N^2$, where $N^1$ is the number of nodes in layer $1$ and $N^2$ is the number of nodes in the second layer. The matrix value $e^{12}_{ij}$ is $1$ if there is a connection and $0$ of there is no connection.

Moreover, in this study, we use a multiplex setting for the two-layer networks. In both layers, the node sets are identical while the edge sets, in general, differ. This implies that the networks can be defined by different adjacency matrices with the same dimensions. It follows that $N^1=N^2$ and, if nodes $i$ and $j$ belong to different  layers, $e_{ij} =1$ if $i=j$. Likewise, in the same layer, self-loops are included, hence $a_{ii}=1$ for all nodes $i$ in the same layer. The undirected nature of the graph makes $a^g_{ij}=a^g_{ji}$. The illustration of a single-layer and a multiplex network can be seen in Figure \ref{fig:graphplot}.

\begin{figure}
\centering
	\includegraphics[width=1\textwidth]{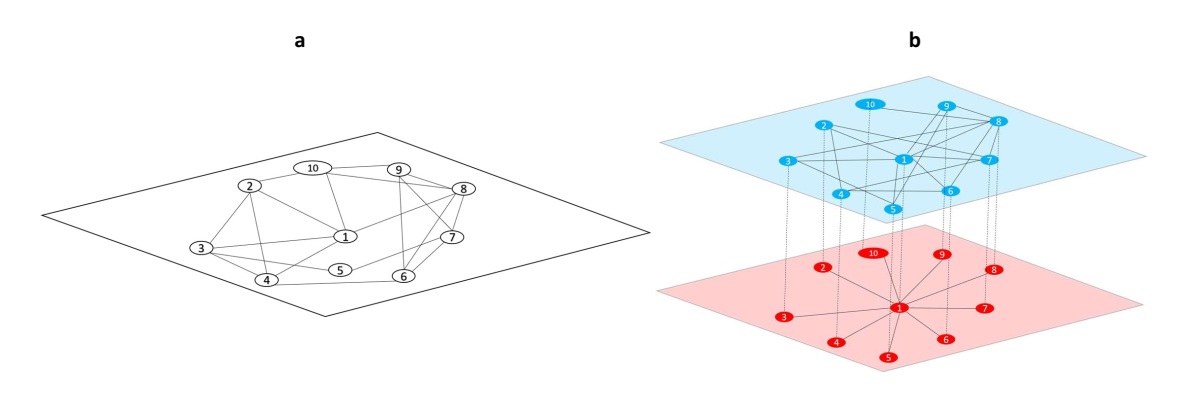}
	\caption[Network Plot]{(a) Single-layer Small-World network. (b) Multilayer network with an Erdős-Rényi network as the first layer (top) and a Star network as the second layer (bottom).}
	\label{fig:graphplot}
\end{figure}

\subsection{Game}
The game considered in this study is a matrix game with the standard payoff structure.
\begin{equation}\label{payoffmatrix}
        \mathcal{P}=\begin{bmatrix}
                     P & Q\\
                     R & S\\
        \end{bmatrix}.
\end{equation}
The rows represent the focal player’s strategy choices, while the columns correspond to those of the opponent. The first row and column denote cooperation, whereas the second row and column denote defection. Accordingly, $P$ represents the payoff received when both players choose to cooperate. Variables $Q$, $R$, and $S$ are interpreted analogously.

For convenience, generally we define $K=P-R$ and $L=Q-S$, representing the relative payoff difference between cooperation and defection under different opponent strategies. Specifically, $K$ denotes the payoff advantage of cooperation over defection when the opponent cooperates. A positive value of $K$ indicates that cooperation is more favourable, whereas a negative value indicates that defection is preferred. Similarly, $L$ represents the corresponding payoff difference when the opponent defects. In here we focus on the Prisoner's Dilemma setting, where $K=k$ and $L=l$ and both $k$ and $l$ are negative, implying that defection yields a higher payoff regardless of the opponent’s strategy. In the following sections, $K$ and $L$ inside differential equation denote general payoff differences and may correspond to any two-strategy game, whereas $k$ and $l$ are reserved for the Prisoner's Dilemma.

The purpose of control is to modify this incentive structure so that cooperation becomes more favourable. Let $\mu_r$ denote the reward for cooperation and $\mu_p$ denote the punishment for defection. The resulting payoff matrix becomes

\begin{equation}\label{payoffmatrixincent}
        \mathcal{P}_i=\begin{bmatrix}
                     P+\mu_r & Q+\mu_r\\
                     R+\mu_p & S+\mu_p\\
        \end{bmatrix}.
\end{equation}

where $\mu_r \geq 0$ and $\mu_p \leq 0$. Although the underlying interaction is represented by a binary matrix game, players may adopt mixed or continuous levels of cooperation. Thus, each node may exhibit partial cooperation rather than choosing exclusively between pure cooperation and pure defection. In this setting, reward and punishment may act simultaneously, motivating the introduction of a single control variable $u=\mu_r-\mu_p$ that captures their combined effect. This follows from
\[
(P+\mu_r)-(R+\mu_p)=(P-R)+(\mu_r-\mu_p)=k+u.
\]
Accordingly, the relative payoff differences become $k+u$ and $l+u$. In the dynamical model, the node-specific control variable $u_i$ represents the incentive mechanism at the $i$th node, where a positive value indicates reward-dominant control and a negative value indicates punishment-dominant control.
\subsection{Dynamical System}
Our work considers strategy evolution through an imitative strategy selection protocol, leading toa replicator dynamics. Before introducing the multilayer setting, the dynamics is considered for a scalar well-mixed system and for the single-layer network model. A summary of the three dynamical systems is provided in Table \ref{tab:dyn}.

\begin{table}[t]
\centering
\caption{Recap of Dynamical System.}
\label{tab:dyn}
\setlength{\tabcolsep}{4pt}
\renewcommand{\arraystretch}{1.08}
\footnotesize
\begin{tabular}{p{2cm} p{10cm} p{3.5cm}}
\hline
System name & Dynamical Equation with Conformity & Game Dynamical Equation \\
\hline
Scalar &
$\dot{x}=((k+u)x+(l+u)(1-x))(1-x)x$ &
$\dot{x}=f_s(x)+u\,h_s(x)$ \\
\hline
Networked (single-layer) &
\(
\begin{aligned}
\dot{\mathbf{x}}
={}&(A^*-I)\mathbf{x}\\
&+\operatorname{diag}\!\left(
A^*\!\left[
(\mathds{1}k+\mathbf{u})\odot\mathbf{x}
+(\mathds{1}l+\mathbf{u})\odot(\mathds{1}-\mathbf{x})
\right]
\right)
\operatorname{diag}(\mathds{1}-\mathbf{x})A^*\mathbf{x}
\end{aligned}
\)&
$\dot{\mathbf{x}}_{gs}=\mathbf{f}_v+G_v\mathbf{u}$ \\
\hline
Networked (Multilayer) &
\(
\begin{aligned}
\dot{\mathbf{x}}
={}&(A_1^*-I)\mathbf{x}\\
&+\operatorname{diag}\!\left(
A_2^*\!\left[
(\mathds{1}k+\mathbf{u})\odot\mathbf{x}
+(\mathds{1}l+\mathbf{u})\odot(\mathds{1}-\mathbf{x})
\right]
\right)
\operatorname{diag}(\mathds{1}-\mathbf{x})A_1^*\mathbf{x}
\end{aligned}
\)&
$\dot{\mathbf{x}}_{gm}=\mathbf{f}_m+G_m\mathbf{u}$ \\
\hline
\end{tabular}
\end{table}
\subsubsection{Scalar Dynamical System}
In the scalar setting, replicator dynamics describe the evolution of strategy proportions through the interaction of the terms $x$ for cooperation and $(1-x)$ for defection, scaled by the payoff-difference parameter $\beta$, which represents the dominant-strategy incentive within the population. A positive value of $\beta$ indicates that cooperation yields a higher payoff, whereas a negative value favours defection.

The variable $x$ denotes the proportion of cooperators in the population, where $x=1$ corresponds to full cooperation and $x=0$ to full defection. When $\beta>0$, the state derivative $\dot{x}$ becomes positive, increasing the proportion of cooperators. Conversely, when $\beta<0$, $\dot{x}$ becomes negative, causing cooperation to decline. The resulting dynamical equation is
\begin{equation}\label{xdotscalar4}
\dot{x}= \beta(1-x)x.
\end{equation}

We define the payoff of cooperation as $\Pi_c=(P+\mu_r)x+(Q+\mu_r)(1-x)$ and the payoff of defection as $\Pi_d=(R+\mu_p)x+(S+\mu_p)(1-x)$. The payoff difference is then given by $\beta=\Pi_c-\Pi_d$, which yields $\beta=(k+u)x+ (l+u)(1-x)$. 

This expression can be rearranged as $(kx+ux+l(1-x)+u-ux)$, which simplifies to
\begin{equation}\label{betascalar}
\beta=(kx+l(1-x)+u).
\end{equation}
We then define $h_s(x)=(1-x)x$ and $f_s(x)=\big(kx+l(1-x)\big)h_s(x)$, such that
\begin{equation}\label{fshs}
\dot{x}=f_s(x)+u \ h_s(x).
\end{equation}

\subsubsection{Networked System}
In the networked setting, whether single-layer or multilayer, the cooperation level of each node $i$ is represented by the strategy vector $\mathbf{x}$, where entries take values in $[0,1]$. This formulation provides a node-level representation of strategic behaviour. In addition to the adjacency matrix $A$, we row-normalised the matrix $A$ to become a row stochastic matrix
\begin{equation}\label{Astar}
 A^*=\operatorname{diag}^{-1}(A\mathds{1})A.
\end{equation}
The matrix $A^*$ transforms the scalar cooperation term into the neighbourhood-averaged state $\Bar{\mathbf{x}}=A^*\mathbf{x}$, which represents the weighted average cooperation level among neighbouring nodes. This formulation captures local strategic interactions while preserving the boundedness of the system within $[0,1]$ \cite{aurachmanPunzo2025strategy}.

As an extension of the scalar well-mixed system, the single-layer network dynamics are formulated by incorporating the matrix $A^*$ and the state vector $\mathbf{x}$ as
\begin{equation}\label{xdotuvector}
        \dot{\mathbf{x}}=(A^*-I)\mathbf{x}+diag( \boldsymbol{\beta}) \,diag(\mathds{1}-\mathbf{x})A^*\mathbf{x}.
\end{equation}
The term $(A^*-I)\mathbf{x}$ represents a conformity mechanism as a linear consensus which, modelled through a Laplacian matrix, drives each node toward alignment to their neighbours' strategies, independently of the payoff-driven game dynamics.

Similar to the scalar well-mixed setting, the interaction term in (\ref{xdotuvector}) retains the multiplicative structure involving $\mathbf{x}$ and $\mathds{1}-\mathbf{x}$, but is mediated by the network matrix $A^*$. This reflects that the spread of cooperative influence is not uniformly distributed across the population, but constrained by the underlying network structure. The scalar payoff difference $\beta$ in the well-mixed model is correspondingly extended to the vector $\boldsymbol{\beta}\in\mathbb{R}^{N}$, allowing each node to have a distinct local incentive profile determined by its network interactions.

In (\ref{xdotuvector}), the payoff structure is modified through the incentive mechanism, yielding the following expression for $\boldsymbol{\beta}$:
\begin{equation}\label{betasinglelayer}
    \boldsymbol{\beta}=A^*\left[ (\mathds{1}k+\mathbf{u}) \odot \mathbf{x} + (\mathds{1}l+\mathbf{u})\odot(\mathds{1}-\mathbf{x})  \right],
\end{equation}
where $\odot$ denotes the Hadamard (element-wise) product. This formulation extends the well-mixed expression in (\ref{betascalar}) by incorporating the network interaction matrix $A^*$. Unlike the scalar setting, where the payoff difference is determined by the aggregate population state, the networked formulation evaluates payoff interactions locally through the proportions of cooperation and defection within each node’s neighbourhood as weighted by $A^*$.

The control variable $\mathbf{u}$ is defined as a vector whose components may differ across nodes, allowing each node to receive a distinct level of incentive in the form of reward or punishment.

For convenience The $\boldsymbol{\beta}$ can also be rewritten as
$\boldsymbol{\beta}=A^*\left[(\mathds{1}k+\mathbf{u})\odot \mathbf{x}+(\mathds{1}l+\mathbf{u})\odot(\mathds{1}-\mathbf{x})
\right]=A^*\left[k\mathbf{x}+l(\mathds{1}-\mathbf{x})+\mathbf{u}\right].$
We define $\mathbf{g}_v:=diag(\mathds{1}-\mathbf{x})(A^*\mathbf{x})$, which represents the diffusion of cooperative strategy across the network. Furthermore, the payoff-driven dynamics without the conformity mechanism are defined as \\
$\mathbf{f}_v
:=
\operatorname{diag}\!\big(\mathbf{g}_v\big)\,
A^*\Big(k\mathbf{x}+l(\mathds{1}-\mathbf{x})\Big).$ with
$G_v(\mathbf{x})
:=
\operatorname{diag}\!\big(\mathbf{g}_v\big)\,A^*$. 
The resulting game dynamics without conformity, but retaining the control input, are given by
\begin{equation}\label{xdotumatrix1}
\dot{\mathbf{x}}_{gs}=(\mathbf{f}_v+G_v\mathbf{u}).
\end{equation}

\subsubsection{Multilayer-networked System}
The multilayer network dynamics are described by
\begin{equation}\label{xdotmultilayer}
        \dot{\mathbf{x}}=(A_1^*-I)\mathbf{x}+diag( \boldsymbol{\beta}_m) \,diag(\mathds{1}-\mathbf{x})A_1^*\mathbf{x},
\end{equation}
which retain a structure similar to the single-layer model in (\ref{xdotuvector}), but distinguish between the network governing social conformity and strategy diffusion and the network governing payoff interactions. The conformity mechanism and strategy diffusion are determined by the first-layer adjacency matrix $A_1^*$, representing the social influence network. The payoff dynamics, in contrast, are evaluated using the second-layer adjacency matrix $A_2^*$, which represent strategic interaction network through
\begin{equation}
    \boldsymbol{\beta}_m=A_2^*\left[ (k\mathds{1}+\mathbf{u}) \odot \mathbf{x} + (l\mathds{1}+\mathbf{u})\odot(\mathds{1}-\mathbf{x})  \right].
\end{equation}
A positive component of $\boldsymbol{\beta}_m$ indicates that cooperation is locally more incentivised, whereas a negative component favours defection. Since $\boldsymbol{\beta}_m$ is node-dependent, the same underlying game parameters may produce different locally preferred strategies across the network, depending on each node’s cooperation level and the cooperation levels of its neighbours. Similar to the single-layer formulation in (\ref{betasinglelayer}), the control vector $\mathbf{u}$ modifies the payoff parameters $k$ and $l$, thereby altering the local game incentives. Unlike the scalar formulation in (\ref{betascalar}), the vector-valued control allows heterogeneous incentive allocation across nodes.

The resulting payoff difference vector $\boldsymbol{\beta}_m$ is then incorporated into the multilayer dynamical system in (\ref{xdotmultilayer}) to determine the evolution of $\mathbf{x}$ over time. While $A_2^*$ governs the payoff interactions that shape the local direction of strategic change, the diffusion of strategy and conformity mechanism is mediated by the separate adjacency matrix $A_1^*$. Consequently, the multilayer structure separates the roles of strategic interaction and social diffusion: $A_2^*$ determines the local payoff structure, whereas $A_1^*$ regulates how strategic changes propagate through neighbouring interactions.

To facilitate the analytical derivation of the optimal control, the multilayer dynamical system is reformulated in a simplified form. The payoff difference vector $\boldsymbol{\beta}_m$ can be rewritten as
$\boldsymbol{\beta}_m=A_2^*\left[(\mathds{1}k+\mathbf{u})\odot \mathbf{x}+(\mathds{1}l+\mathbf{u})\odot(\mathds{1}-\mathbf{x})
\right]=A^*\left[k\mathbf{x}+l(\mathds{1}-\mathbf{x})+\mathbf{u}.\right]$
We further define the diffusion term $\mathbf{g}_m:=(\mathds{1}-\mathbf{x})\odot(A_1^*\mathbf{x}).$\\
The multilayer payoff-driven dynamics without the conformity mechanism are then defined as $\mathbf{f}_m
:=\operatorname{diag}\!\big(\mathbf{g}_m\big)\,
A_2^*\Big(k\mathbf{x}+l(\mathds{1}-\mathbf{x})\Big)$ and
$G_m(\mathbf{x})
:=
\operatorname{diag}\!\big(\mathbf{g}_m(\mathbf{x})\big)\,A_2^*.$
Accordingly, the incentive-controlled multilayer game dynamics without conformity can be expressed as
\begin{equation}\label{xdotumatrix2}
\dot{\mathbf{x}}_{gm}=(\mathbf{f}_m+G_m\mathbf{u}).
\end{equation}

\section{Formal Statement}
This section will present the analytical solution for each dynamical system, recapped on table \ref{tab:dyn}. The solution will be presented by several formal statements, which are recapitulated in table \ref{tab:formal_recap}. It will begin with a lemma proving the defective equilibrium of the Prisoner's Dilemma dynamical system without control, and then prove the optimal solution that drives the system toward cooperation.
\begin{table}[t]
\centering
\caption{Recap of formal statements.}
\label{tab:formal_recap}
\begin{tabular}{p{4.0cm} p{12.4cm}}
\hline
\textbf{Statement} & \textbf{Main result} \\
\hline
Lemma~\ref{scalarequilibrium} Equilibrium of Scalar System&
For $\dot{x}=(Kx+L(1-x))(1-x)x$, the system moves toward cooperation if $K>0$ and $L>0$, and toward defection if $K<0$ and $L<0$. \\
\hline
Theorem \ref{optscalar} Optimal Control for scalar system&
For the Prisoner's Dilemma setting, a stationary-HJB candidate feedback is
$u^*(x)=-2\big(kx+l(1-x)\big)$, make the system move toward cooperation. \\
\hline
Lemma~\ref{vectorequilibrium} Equilibrium of Networked System &
For the single-layer networked system in equation \eqref{xdotuvector}, the direction of motion is toward cooperation when $K>0,L>0$ and toward defection when $K<0,L<0$. \\
\hline
Theorem \ref{optinvertible} Optimal Control for networked system  with invertible parameter &
For $\dot{\mathbf{x}}_{gs}=\mathbf{f}_v+G_v\mathbf{u}$, if $G_vG_v^\top$ is invertible, a stationary-HJB candidate control is
$\mathbf{u}^*=-2\,G_v^\top(G_vG_v^\top)^{+}\mathbf{f}_v$, make the system moves toward cooperation. \\
\hline
Lemma~\ref{multilayerequilibrium} Equilibrium of Multi-layer Networked System&
For the single-layer networked system in equation \eqref{xdotmultilayer}, the system moves toward cooperation for $K>0,L>0$ and toward defection for $K<0,L<0$. \\
\hline
Theorem \ref{optmultiinver} Optimal Control for multilayer networked system  with invertible parameter &
For $\dot{\mathbf{x}}_{gm}=\mathbf{f}_m+G_m\mathbf{u}$, if $G_mG_m^\top$ is invertible, a stationary-HJB candidate control is
$\mathbf{u}^*=-2\,G_m^\top(G_mG_m^\top)^{-1}\mathbf{f}_m$, and make the system moves toward cooperation. \\
\hline
\end{tabular}
\end{table}

The first lemma \ref{scalarequilibrium} proves that the scalar system with the Prisoner's Dilemma setting will make defection the only stable equilibrium.
\subsection{Optimal Incentive for the Scalar System}
\begin{lemma}\label{scalarequilibrium}

The scalar dynamical system 
\begin{equation}\label{bigkl}
    \dot{x}= (Kx+ L(1-x))(1-x)x.
\end{equation}
 will move toward cooperation if $K>0$ and $L>0$, and move towards defection if it follows Prisoner's Dilemma setting, which has $K<0$ and $L<0$
\end{lemma}
\begin{proof}
To establish the asymptotic stability of the equilibrium $x=1$ ($x=0$), consider the Lyapunov function $V=\frac{1}{2}(1-x)^2$ ($V=\frac{1}{2}x^2$), The function $V$ is positive or equals zero at $x=1$ ($x=0$).  The time derivative is given by $\dot{V}=(x-1)\dot{x}$ ($\dot{V}=x\dot{x}$). Substituting the system dynamics yields $\dot{V}=(x-1)(Kx+ L(1-x))(1-x)x$ ($\dot{V}=x(Kx+ L(1-x))(1-x)x$).

Since $1-x$ and ($x$) is non-negative, the sign of $\dot{V}$ is determined by $(x-1)(Kx+ L(1-x)$ ($x(Kx+ L(1-x)$). Therefore, since $x-1\leq0$ ($x\geq0$) multiplied by $(x-1)(Kx+ L(1-x)$ that have $K>0$ and $L>0$ ($K<0$ and $L<0$) will make every term in the summation nonpositive, yielding $\dot{V}\leq0$. It follows that $K>0$ and $L>0$ ($K<0$ and $L<0$) drive the system toward cooperation (defection), with the equilibrium $x=1$ ($x=0$) being asymptotically stable.
\end{proof}

Using the statement from Lemma \ref{scalarequilibrium}, the following theorem presents the optimal control solution with fixed end state and free end time, using the Hamilton-Jacobi-Bellman equation method to find the optimal control strategy. The fixed endpoint will be some point near the cooperation.

\begin{theorem}\label{optscalar}
Consider the optimal control problem associated with the Prisoner's Dilemma ($k<0$ and $l<0$) in the scalar dynamical system given by equation \ref{fshs}. The control input that minimises the performance index $J=\int_0^{t_f}\frac12 u^2,dt$, subject to the terminal cooperation constraint $x(t_f)=1-\theta$ and free terminal time, is given by $u^*(x)=-2\big(kx+l(1-x)\big)$.
\end{theorem}

\begin{proof}
The form of optimal control problem will become
\begin{equation}
\min_{u}\quad
J
=\int_{0}^{t_f}\frac{1}{2}u^2 dt,
\end{equation}
\begin{equation}
\text{s.t.}\;
\begin{cases}
\dot{x}=f_s(x)+u\,h_s(x),\\[2pt]
x(0)=x_0,\\[2pt]
x(t_f)=1-\theta.
\end{cases}
\end{equation}

The Hamiltonian is $H(x,u,\frac{dJ^*}{dx})=\frac{1}{2}u^2+\frac{dJ^*}{dx}\dot{x}$ or we can say 
\begin{equation}
H=\frac{1}{2}u^2+\frac{dJ^*}{dx}(f_s(x)+u\,h_s(x)). 
\end{equation}

Then the necessary condition that the optimal control must satisfy is $\frac{\partial H}{\partial u}=u+\frac{dJ^{*}}{dx}h_s(x)=0$. From there we found that optimal control $u^*$ is 
\begin{equation}\label{optimalus}
    u^*=-\frac{dJ^{*}}{dx}h_s(x). 
\end{equation}

The stationary Hamilton-Jacobi Bellman equation can be written as
\begin{equation}
0=\frac{\partial J^{}}{\partial t}+H(x,u,\frac{dJ^*}{dx}).
\end{equation}
Assuming $J^*$ does not depend on time $t$, which means that $\frac{\partial J^{}}{\partial t}=0$. Thus the HJB equation can be written as
\begin{equation}
0=\frac{1}{2}u^2+\frac{dJ^*}{dx}(f_s(x)+u\,h_s(x)).
\end{equation}
By inserting $u^*$ into the HJB equation 
\begin{equation}
0=\frac{1}{2}\big(u^*\big)^2+\frac{dJ^*}{dx}(f_s(x)+u^*h_s(x)).
\end{equation}

Since $u^*=-\frac{dJ^*}{dx}h_s(x)$, we have the HJB equation will become
\begin{equation}
0=\frac{1}{2}\left(\frac{dJ^*}{dx}\right)^2 h_s(x)^2
+\frac{dJ^*}{dx}\left(f_s(x)-\frac{dJ^*}{dx}h_s(x)^2\right)
=\frac{dJ^*}{dx}\,f_s(x)-\frac{1}{2}\left(\frac{dJ^*}{dx}\right)^2h_s(x)^2.
\end{equation}
Factoring $\frac{dJ^*}{dx}$ gives
\begin{equation}
0=\frac{dJ^*}{dx}\left(f_s(x)-\frac{1}{2}\frac{dJ^*}{dx}h_s(x)^2\right),
\end{equation}
which yields two candidate solutions:
\begin{equation}
\left(\frac{dJ^*}{dx}\right)_1=0
\qquad\text{or}\qquad
\left(\frac{dJ^*}{dx}\right)_2=\frac{2f_s(x)}{h_s(x)^2}
=\frac{2\big(kx+l(1-x)\big)}{(1-x)x}, \quad x\in(0,1).
\end{equation}
From $u^*(x)=-\frac{dJ^*}{dx}h_s(x)$ we obtain the corresponding controls
\begin{equation}\label{uscalar}
u_1^*(x)=0,
\qquad
u_2^*(x)=-\left(\frac{dJ^*}{dx}\right)_2 h_s(x)
=-2\big(kx+l(1-x)\big).
\end{equation}
If $u=0$, the dynamical system in \ref{fshs} in the form of equation \ref{bigkl} will make $K<0$ and $L<0$, which will form a Prisoner's Dilemma setting. Based on the lemma \ref{scalarequilibrium}, the system drifts toward defection; thus, $u>0$ and $0<x<1$ should be satisfied. Therefore the $\frac{dJ^*}{dx}$ cannot be $0$ since, based on equation \ref{optimalus}, it will make $u=0$. Therefore, we use $(\frac{dJ^*}{dx})_2$, which corresponds to $u_2=-2\big(kx+l(1-x)\big)$, which becomes the optimal control for the system. This optimal control will modify the dynamical equation \ref{xdotscalar4} while considering $\beta$ in equation \ref{betascalar}  to become $\dot{x}=(kx+l(1-x)+u_2)(1-x)x=(kx+l(1-x)-2\big(kx+l(1-x)\big))(1-x)x=-(kx+l(1-x))(1-x)x$ or
\begin{equation}\label{optimizedxdotscal}
    \dot{x}=(-kx+-l(1-x))(1-x)x
\end{equation}
In the Prisoner's Dilemma setting, since $K=-k$ and $L=-l$, $k<0$ and $l<0$ imply $K>0$ and $L> 0$. Based on Lemma \ref{scalarequilibrium}, the system in equation \ref{fshs} will move toward cooperation.
\end{proof}
The Theorem \ref{optscalar} shows that the optimal control strategy is the function of current cooperation intensity $x$, the payoff matrix variable $k$ and $l$. It shows that the amount of incentive will change over time, adapting to the system's current cooperation intensity.

\subsection{Optimal Incentive for Single Layer Networked System}
The next step is to find the analytical solution of optimal control in a single-layer networked system. We begin with formal statement that explain the equilibrium of the system as presented in lemma
\begin{lemma}\label{vectorequilibrium}
Consider the dynamical system
\begin{equation}\label{bigklvector}
        \dot{\mathbf{x}}=(A^*-I)\mathbf{x}+diag(A^*\left[ K\mathbf{x} + L(\mathds{1}-\mathbf{x})  \right]) \,diag(\mathds{1}-\mathbf{x})A^*\mathbf{x}.
\end{equation}
If $K>0$ and $L>0$, the system promotes cooperation. Conversely, under the Prisoner's Dilemma setting, where $K<0$ and $L<0$, the system promotes defection.
\end{lemma}
\begin{proof}
To establish the asymptotic stability of the equilibrium $\mathbf{x}=\mathds{1}$ ($\mathbf{x}=\mathbf{0}$), consider the Lyapunov function $V=\pi^T(\mathds{1}-\mathbf{x})$ ($V=\pi^T\mathbf{x}$), where $\pi$ is the left eigenvector of $A^*$ associated with the eigenvalue $1$, satisfying $\pi^T A^*=\pi^T$. The function $V$ is positive definite and equals zero at $\mathbf{x}=\mathds{1}$ ($\mathbf{x}=\mathbf{0}$). The time derivative is given by $\dot{V}=-\pi^T\dot{\mathbf{x}}$ ($\dot{V}=\pi^T\dot{\mathbf{x}}$). Substituting the system dynamics yields $\dot{V}=-\pi^T\left((A^*-I)\mathbf{x}+diag(\boldsymbol{\beta}),diag(\mathds{1}-\mathbf{x})A^*\mathbf{x}\right)$ ($\dot{V}=\pi^T\left((A^*-I)\mathbf{x}+diag(\boldsymbol{\beta}),diag(\mathds{1}-\mathbf{x})A^*\mathbf{x}\right)$).

The conformity term satisfies $\pi^T(A^*-I)\mathbf{x}=0$ since $\pi^T A^*=\pi^T$, causing the conformity mechanism to vanish. If $K>0$ and $L>0$ ($K<0$ and $L<0$), then $\boldsymbol{\beta}\succ0$ ($\boldsymbol{\beta}\prec0$). Since $\pi\succ0$, $\mathds{1}-\mathbf{x}\succeq0$, and $A^*\mathbf{x}\succeq0$, the sign of $\dot{V}$ is determined by $\boldsymbol{\beta}$. Therefore, if $\boldsymbol{\beta}\succ0$ ($\boldsymbol{\beta}\prec0$), every term in the summation is nonpositive, yielding $\dot{V}\leq0$. It follows that $K>0$ and $L>0$ ($K<0$ and $L<0$) drive the system toward cooperation (defection), with the equilibrium $\mathbf{x}=\mathds{1}$ ($\mathbf{x}=\mathbf{0}$) being asymptotically stable.
\end{proof}

Understanding that the Prisoner's Dilemma will make the system move toward defection, the next Theorem \ref{optinvertible} will explain 
\begin{theorem}\label{optinvertible}
Consider the optimal control problem associated with the Prisoner's Dilemma ($k<0$ and $l<0$) in the single-layer dynamical system given by equation \ref{xdotumatrix1}. If $\Big(G_vG_v^\top\Big)$ is invertible, the control input that minimises the performance index $J
=\int_{0}^{t_f}\frac{1}{2}\,\mathbf{u}^{\top}\mathbf{u}\,dt$, subject to the terminal cooperation constraint $\mathbf{x}(t_f)=1-\theta$ and free terminal time, is given by $\mathbf{u}^*=-2\,G_v^\top \Big(G_vG_v^\top\Big)^{-1}\mathbf{f}_v$.
\end{theorem}
\begin{proof}
The form of optimal control will become
\begin{equation}\label{optproblemSL}
\min_{\mathbf{u}}\quad
J
=\int_{0}^{t_f}\frac{1}{2}\,\mathbf{u}^{\top}\mathbf{u}\,dt,
\end{equation}
\begin{equation}
\text{s.t.}\;
\begin{cases}
\dot{\mathbf{x}}_{gs}=\mathbf{f}_v+G_v\mathbf{u},\\
\mathbf{x}(0)=\mathbf{x}_0,\\[2pt]
\mathbf{x}(t_f)=\mathds{1}-\boldsymbol{\theta}.
\end{cases}
\end{equation}

\begin{equation}
\mathcal{H}\!
=
\frac12\,\mathbf{u}^\top \mathbf{u}
+(\nabla_{\mathbf{x}}J^*)^\top\left(\mathbf{f}_v+G_v\mathbf{u}\right).
\end{equation}

The stationarity condition with respect to $\mathbf{u}$ is
\begin{equation}
\frac{\partial \mathcal{H}}{\partial \mathbf{u}}
=
\mathbf{u}+G_v^\top \nabla_{\mathbf{x}}J^*
,
\end{equation}

Then the necessary condition that the optimal control must satisfy is $\frac{\partial \mathcal{H}}{\partial \mathbf{u}}=\mathbf{u}+\nabla_{\mathbf{x}}J^*(\mathbf{x})^{\top}\mathbf{h}_s(x)=0$. From there we found that optimal control $\mathbf{u}^*$ is
\begin{equation}\label{unablainvertable}
\mathbf{u}^*=-G_v^\top \nabla_{\mathbf{x}}J^*.
\end{equation}

Let $J^*(\mathbf{x},t)$ denote the value function. The Hamilton--Jacobi--Bellman equation is
\begin{equation}\label{HJB_full_timeSL}
0=
\frac{\partial J^*(\mathbf{x},t)}{\partial t}
+\mathcal{H}\!\left(\mathbf{x},\mathbf{u},t,\nabla_{\mathbf{x}}J^*\right),
\end{equation}
or equivalently,
\begin{equation}\label{HJB_full_time_expandedSL}
0=
\frac{\partial J^*(\mathbf{x},t)}{\partial t}
+
\frac{1}{2}\,\mathbf{u}^{*\top} \mathbf{u}^*
+\big(\nabla_{\mathbf{x}}J^*(\mathbf{x},t)\big)^\top\!\left(\mathbf{f}_v+G_v\mathbf{u}^*\right).
\end{equation}

Assume that $J^*(\mathbf{x})$ does not explicitly depend on $t$), so that
\begin{equation}\label{Jt_zeroSL}
\frac{\partial J^*(\mathbf{x},t)}{\partial t}=0.
\end{equation}
Then the stationary HJB equation will become
\begin{equation}\label{HJB_stationary_cleanSL}
0=
\frac{1}{2}\,\mathbf{u}^{*\top} \mathbf{u}^*
+\big(\nabla_{\mathbf{x}}J^*(\mathbf{x},t)\big)^\top\!(\mathbf{f}_v+G_v\mathbf{u}^*).
\end{equation}

We now substitute $\mathbf{u}^*$ based on equation \ref{unablainvertable}. First,
\begin{equation}\label{uTu_sub_nopSL}
(\mathbf{u}^*)^\top\mathbf{u}^*
=
\big(-G_v^\top\nabla_{\mathbf{x}}J^*\big)^\top\big(-G_v^\top\nabla_{\mathbf{x}}J^*\big)
=
\big(\nabla_{\mathbf{x}}J^*\big)^\top G_vG_v^\top \nabla_{\mathbf{x}}J^*.
\end{equation}
Second,
\begin{equation}\label{gradGusub_nopSL}
\big(\nabla_{\mathbf{x}}J^*\big)^\top G_v\mathbf{u}^*
=
\big(\nabla_{\mathbf{x}}J^*\big)^\top G_v\big(-G_v^\top\nabla_{\mathbf{x}}J^*\big)
=
-\big(\nabla_{\mathbf{x}}J^*\big)^\top G_vG_v^\top \nabla_{\mathbf{x}}J^*.
\end{equation}
Therefore the HJB equation will become,
\begin{align}\label{Phi_star_expand_nopSL}
0
&=
\frac12\,(\mathbf{u}^*)^\top \mathbf{u}^*
+\big(\nabla_{\mathbf{x}}J^*\big)^\top \mathbf{f}_v
+\big(\nabla_{\mathbf{x}}J^*\big)^\top G_v\mathbf{u}^* \nonumber\\
&=
\frac12\,\big(\nabla_{\mathbf{x}}J^*\big)^\top G_vG_v^\top \nabla_{\mathbf{x}}J^*
+\big(\nabla_{\mathbf{x}}J^*\big)^\top \mathbf{f}_v
-\big(\nabla_{\mathbf{x}}J^*\big)^\top G_vG_v^\top \nabla_{\mathbf{x}}J^*.
\end{align}

Finally, we obtain the reduced stationary HJB equation
\begin{equation}\label{HJB_reduced_nopSL}
0=
\big(\nabla_{\mathbf{x}}J^*\big)^\top \mathbf{f}_v
-\frac12\,\big(\nabla_{\mathbf{x}}J^*\big)^\top
G_vG_v^\top
\big(\nabla_{\mathbf{x}}J^*\big).
\end{equation}

From \eqref{HJB_reduced_nopSL}, we have
\begin{equation}\label{HJB_factor_grad_cleanGSL}
0=
\big(\nabla_{\mathbf{x}}J^*\big)^\top \mathbf{f}_v
-\frac12\,\big(\nabla_{\mathbf{x}}J^*\big)^\top
\Big(G_vG_v^\top\Big)
\big(\nabla_{\mathbf{x}}J^*\big)
=
\big(\nabla_{\mathbf{x}}J^*\big)^\top
\left(
\mathbf{f}_v
-\frac12\,\Big(G_vG_v^\top\Big)\big(\nabla_{\mathbf{x}}J^*\big)
\right).
\end{equation}

This yields two stationary candidate solutions for $\nabla_{\mathbf{x}}J^*$:
\begin{equation}\label{grad_candidates_cleanGSL}
\left(\nabla_{\mathbf{x}}J^*\right)_1=\mathbf{0},
\qquad
\left(\nabla_{\mathbf{x}}J^*\right)_2
=
2\,\Big(G_vG_v^\top\Big)^{-1}\mathbf{f}_v.
\end{equation}

From equation \eqref{unablainvertable}, the corresponding candidate controls are
\begin{equation}\label{u1veccandidate}
\mathbf{u}_1^*=-G_v^\top (\nabla_{\mathbf{x}}J^*)_1=-G_v^\top (0)=0
\end{equation}
and
\begin{equation}\label{u2veccandidate}
\mathbf{u}_2^*
=
-G_v^\top \left(\nabla_{\mathbf{x}}J^*\right)_2
=
-2\,G_v^\top \Big(G_vG_v^\top\Big)^{-1}\mathbf{f}_v.
\end{equation}

If $\mathbf{u}=0$, the dynamical system in \ref{xdotumatrix1} and \ref{xdotuvector} in the form of equation \ref{bigklvector} will make $K<0$ and $L<0$, which will form a Prisoner's Dilemma setting. Based on the lemma \ref{vectorequilibrium}, the system drifts toward defection; thus to satisfy optimal control problem constraint on the equation \ref{optproblemSL}, $\mathbf{u}\neq0$ should be satisfied. Therefore the $\nabla_{\mathbf{x}}J^*$ cannot be $\mathbf{0}$ since, based on equation \ref{u1veccandidate}, it will make $u=0$. Therefore, we use $\left(\nabla_{\mathbf{x}}J^*\right)_2
$, which corresponds to $-2\,G_v^\top \Big(G_vG_v^\top\Big)^{-1}\mathbf{f}_v$, which becomes the optimal control for the system. This optimal control will modify the dynamical equation \ref{xdotumatrix2} to become $\dot{\mathbf{x}}_{gs}=(\mathbf{f}_v+G_v\mathbf{u}_2^*)=(\mathbf{f}_v-2G_v\,G_v^\top \Big(G_vG_v^\top\Big)^{-1}\mathbf{f}_v)=\mathbf{f}_v-2\mathbf{f}_v=-\mathbf{f}_v=$ and if added with the conformity mechanism, the dynamical equation will become
\begin{equation}\label{optimizedxdotsinglay}
    \Tilde{\dot{\mathbf{x}}}=(A_1^*-I)\mathbf{x}+diag(A_2^*\left[ (-k) \mathbf{x} + (-l)(\mathds{1}-\mathbf{x})  \right]) \,diag(\mathds{1}-\mathbf{x})A_1^*\mathbf{x}.
\end{equation}
In the Prisoner's Dilemma setting, $k<0$ and $l<0$ imply $K>0$ and $L> 0$. Based on Lemma \ref{vectorequilibrium}, the system in equation \ref{xdotuvector} will move toward cooperation.
\end{proof}
\subsection{Optimal Incentive for Multi-Layer Networked System}
\begin{lemma}\label{multilayerequilibrium}
The dynamical system
\begin{equation}\label{bigklvectormulti}
\dot{\mathbf{x}}=(A_1^*-I)\mathbf{x}+diag(\boldsymbol{\beta}_m) ,diag(\mathds{1}-\mathbf{x})A_1^*\mathbf{x}.
\end{equation}
with $\boldsymbol{\beta}_m=A_2^*\left[ K\mathbf{x} + L(\mathds{1}-\mathbf{x}) \right]$ promotes cooperation if $K>0$ and $L>0$, and promotes defection under the Prisoner's Dilemma setting, where $K<0$ and $L<0$.
\end{lemma}
\begin{proof}
To establish the asymptotic stability of the equilibrium $\mathbf{x}=\mathds{1}$ ($\mathbf{x}=\mathbf{0}$), consider the Lyapunov function $V=\pi^T(\mathds{1}-\mathbf{x})$ ($V=\pi^T\mathbf{x}$), where $\pi$ is the left eigenvector of $A_1^*$ associated with the eigenvalue $1$, satisfying $\pi^T A_1^*=\pi^T$. The function $V$ is positive definite and equals zero at $\mathbf{x}=\mathds{1}$ ($\mathbf{x}=\mathbf{0}$). The time derivative is given by $\dot{V}=-\pi^T\dot{\mathbf{x}}$ ($\dot{V}=\pi^T\dot{\mathbf{x}}$). Substituting the system dynamics yields $\dot{V}=-\pi^T\left((A_1^*-I)\mathbf{x}+diag(\boldsymbol{\beta}),diag(\mathds{1}-\mathbf{x})A_1^*\mathbf{x}\right)$ ($\dot{V}=\pi^T\left((A_1^*-I)\mathbf{x}+diag(\boldsymbol{\beta}),diag(\mathds{1}-\mathbf{x})A_1^*\mathbf{x}\right)$).

The conformity term satisfies $\pi^T(A_1^*-I)\mathbf{x}=0$ since $\pi^T A_1^*=\pi^T$, causing the conformity contribution to vanish. If $K>0$ and $L>0$ ($K<0$ and $L<0$), then $\boldsymbol{\beta}\succ0$ ($\boldsymbol{\beta}\prec0$). Since $\pi\succ0$, $\mathds{1}-\mathbf{x}\succeq0$, and $A_1^*\mathbf{x}\succeq0$, the sign of $\dot{V}$ is determined by $\boldsymbol{\beta}$. Therefore, if $\boldsymbol{\beta}\succ0$ ($\boldsymbol{\beta}\prec0$), every term in the summation is nonpositive, yielding $\dot{V}\leq0$. It follows that $K>0$ and $L>0$ ($K<0$ and $L<0$) promote cooperation (defection), with the equilibrium $\mathbf{x}=\mathds{1}$ ($\mathbf{x}=\mathbf{0}$) being asymptotically stable.
\end{proof}

\begin{theorem}\label{optmultiinver}
Consider the optimal control problem associated with the Prisoner's Dilemma ($k<0$ and $l<0$) in the multilayer dynamical system given by equation \ref{xdotumatrix2}. If $\Big(G_mG_m^\top\Big)$ is invertible, the control input that minimises the performance index $J
=\int_{0}^{t_f}\frac{1}{2}\,\mathbf{u}^{\top}\mathbf{u}\,dt,$, subject to the terminal cooperation constraint $\mathbf{x}(t_f)=1-\theta$ and free terminal time, is given by $\mathbf{u}^*=-2\,G_m^\top \Big(G_mG_m^\top\Big)^{-1}\mathbf{f}_m$.
\end{theorem}
\begin{proof}
The form of optimal control will become
\begin{equation}\label{optconML}
\min_{\mathbf{u}}\quad
J
=\int_{0}^{t_f}\frac{1}{2}\,\mathbf{u}^{\top}\mathbf{u}\,dt,
\end{equation}
\begin{equation}
\text{s.t.}\;
\begin{cases}
\dot{\mathbf{x}}_{gm}=\mathbf{f}_m+G_m\mathbf{u},\\
\mathbf{x}(0)=\mathbf{x}_0,\\[2pt]
\mathbf{x}(t_f)=\mathds{1}-\boldsymbol{\theta}.
\end{cases}
\end{equation}

Assume a stationary value function $J^*(\mathbf{x})$ so that $\frac{\partial J^*}{\partial t}=0$.

\begin{equation}
\mathcal{H}\!
=
\frac12\,\mathbf{u}^\top \mathbf{u}
+(\nabla_{\mathbf{x}}J^*)^\top\left(\mathbf{f}_m+G_m\mathbf{u}\right).
\end{equation}

The stationarity condition with respect to $\mathbf{u}$ is
\begin{equation}
\frac{\partial \mathcal{H}}{\partial \mathbf{u}}
=
\mathbf{u}+G_m^\top \nabla_{\mathbf{x}}J^*
.
\end{equation}

Then the necessary condition that the optimal control must satisfy is $\frac{\partial \mathcal{H}}{\partial \mathbf{u}}=\mathbf{u}+\nabla_{\mathbf{x}}J^*(\mathbf{x})^{\top}\mathbf{h}_s(x)=0$. From there we found that optimal control $\mathbf{u}^*$ is
\begin{equation}\label{unablamultilayer}
\mathbf{u}^*=-G_m^\top \nabla_{\mathbf{x}}J^*.
\end{equation}

Let $J^*(\mathbf{x},t)$ denote the value function. The Hamilton--Jacobi--Bellman equation is
\begin{equation}\label{HJB_full_timeML}
0=
\frac{\partial J^*(\mathbf{x},t)}{\partial t}
+\mathcal{H}\!\left(\mathbf{x},\mathbf{u},t,\nabla_{\mathbf{x}}J^*\right),
\end{equation}
or equivalently,
\begin{equation}\label{HJB_full_time_expandedML}
0=
\frac{\partial J^*(\mathbf{x},t)}{\partial t}
+
\frac{1}{2}\,\mathbf{u}^{*\top} \mathbf{u}^*
+\big(\nabla_{\mathbf{x}}J^*(\mathbf{x},t)\big)^\top\!\left(\mathbf{f}_m+G_m\mathbf{u}^*\right).
\end{equation}

Assume a stationary value function $J^*(\mathbf{x})$ does not explicitly depend on $t$, so that
\begin{equation}\label{Jt_zeroML}
\frac{\partial J^*(\mathbf{x},t)}{\partial t}=0.
\end{equation}
Then \eqref{HJB_full_time_expandedML} reduces to the stationary HJB equation
\begin{equation}\label{HJB_stationary_cleanML}
0=
\frac{1}{2}\,\mathbf{u}^{*\top} \mathbf{u}^*
+\big(\nabla_{\mathbf{x}}J^*(\mathbf{x},t)\big)^\top\!(\mathbf{f}_m+G_m\mathbf{u}^*).
\end{equation}

We now substitute $\mathbf{u}^*$ based on equation \ref{unablamultilayer}. First,
\begin{equation}\label{uTu_sub_nopML}
(\mathbf{u}^*)^\top\mathbf{u}^*
=
\big(-G_m^\top\nabla_{\mathbf{x}}J^*\big)^\top\big(-G_m^\top\nabla_{\mathbf{x}}J^*\big)
=
\big(\nabla_{\mathbf{x}}J^*\big)^\top G_mG_m^\top \nabla_{\mathbf{x}}J^*.
\end{equation}
Second,
\begin{equation}\label{gradGusub_nopML}
\big(\nabla_{\mathbf{x}}J^*\big)^\top G_m\mathbf{u}^*
=
\big(\nabla_{\mathbf{x}}J^*\big)^\top G_m\big(-G_m^\top\nabla_{\mathbf{x}}J^*\big)
=
-\big(\nabla_{\mathbf{x}}J^*\big)^\top G_mG_m^\top \nabla_{\mathbf{x}}J^*.
\end{equation}
Therefore the HJB equation will become,
\begin{align}\label{Phi_star_expand_nopML}
0
&=
\frac12\,(\mathbf{u}^*)^\top \mathbf{u}^*
+\big(\nabla_{\mathbf{x}}J^*\big)^\top \mathbf{f}_m
+\big(\nabla_{\mathbf{x}}J^*\big)^\top G_m\mathbf{u}^* \nonumber\\
&=
\frac12\,\big(\nabla_{\mathbf{x}}J^*\big)^\top G_mG_m^\top \nabla_{\mathbf{x}}J^*
+\big(\nabla_{\mathbf{x}}J^*\big)^\top \mathbf{f}_m
-\big(\nabla_{\mathbf{x}}J^*\big)^\top G_mG_m^\top \nabla_{\mathbf{x}}J^* .
\end{align}

Finally, we obtain the reduced stationary HJB equation
\begin{equation}\label{HJB_reduced_nopML}
0=
\big(\nabla_{\mathbf{x}}J^*\big)^\top \mathbf{f}_m
-\frac12\,\big(\nabla_{\mathbf{x}}J^*\big)^\top
G_mG_m^\top
\big(\nabla_{\mathbf{x}}J^*\big).
\end{equation}

From \eqref{HJB_reduced_nopML}, we have
\begin{equation}\label{HJB_factor_grad_cleanGML}
0=
\big(\nabla_{\mathbf{x}}J^*\big)^\top \mathbf{f}_m
-\frac12\,\big(\nabla_{\mathbf{x}}J^*\big)^\top
\Big(G_mG_m^\top\Big)
\big(\nabla_{\mathbf{x}}J^*\big)
=
\big(\nabla_{\mathbf{x}}J^*\big)^\top
\left(
\mathbf{f}_m
-\frac12\,\Big(G_mG_m^\top\Big)\big(\nabla_{\mathbf{x}}J^*\big)
\right).
\end{equation}

This yields two stationary candidate solutions for $\nabla_{\mathbf{x}}J^*$:
\begin{equation}\label{grad_candidates_cleanGML}
\left(\nabla_{\mathbf{x}}J^*\right)_1=\mathbf{0},
\qquad
\left(\nabla_{\mathbf{x}}J^*\right)_2
=
2\,\Big(G_mG_m^\top\Big)^{-1}\mathbf{f}_m.
\end{equation}

From equation \eqref{unablamultilayer}, the corresponding candidate controls are
\begin{equation}\label{u1veccandidateMul}
\mathbf{u}_1^*=-G_m^\top (\nabla_{\mathbf{x}}J^*)_1=-G_m^\top (0)=0
\end{equation}
and
\begin{equation}\label{u2veccandidateMul}
\mathbf{u}_2^*
=
-G_m^\top \left(\nabla_{\mathbf{x}}J^*\right)_2
=
-2\,G_m^\top \Big(G_mG_m^\top\Big)^{-1}\mathbf{f}_m.
\end{equation}

If $\mathbf{u}=0$, the dynamical system in \ref{xdotumatrix2} and \ref{xdotmultilayer} in the form of equation \ref{bigklvectormulti} will make $K<0$ and $L<0$, which will form a Prisoner's Dilemma setting. Based on the lemma \ref{multilayerequilibrium}, the system drifts toward defection; thus, to satisfy the constraint from the optimal problem on equation \ref{optconML}, $\mathbf{u}\neq0$ should be satisfied. Therefore the $\nabla_{\mathbf{x}}J^*$ cannot be $\mathbf{0}$ since, based on equation \ref{u1veccandidateMul}, it will make $u=0$. Therefore, we use $\left(\nabla_{\mathbf{x}}J^*\right)_2
$, which corresponds to $-2\,G_m^\top \Big(G_mG_m^\top\Big)^{-1}\mathbf{f}_m,$, which becomes the optimal control for the system. This optimal control will modify the dynamical equation \ref{xdotumatrix2} to become $\dot{\mathbf{x}}_{gm}=(\mathbf{f}_m+G_m\mathbf{u}_2^*)=(\mathbf{f}_m-2G_m\,G_m^\top \Big(G_mG_m^\top\Big)^{-1}\mathbf{f}_m)=\mathbf{f}_m-2\mathbf{f}_m=-\mathbf{f}_m$ and if added by the conformity part as dampening to become
\begin{equation}\label{optimizedxdotmulay}
    \dot{\mathbf{x}}=(A_1^*-I)\mathbf{x}+diag(A_2^*\left[ (-k) \mathbf{x} + (-l)(\mathds{1}-\mathbf{x})  \right]) \,diag(\mathds{1}-\mathbf{x})A_1^*\mathbf{x}.
\end{equation}
In the Prisoner's Dilemma setting, $k<0$ and $l<0$ imply $K>0$ and $L> 0$. Based on Lemma \ref{multilayerequilibrium}, the system in equation \ref{xdotmultilayer} will move toward cooperation.
\end{proof}
\section{Numerical Simulations}

This section presents numerical simulations to validate the analytical results derived in the previous section. The scalar system simulations assume a well-mixed population. The networked system simulations use 10-node networks with various structures, including Erdős-Rényi, Small World, and star networks (Figure \ref{fig:graphplot}). The multilayer system simulations consider combinations of these network structures.

\begin{figure}
\centering
	\includegraphics[width=.5\textwidth]{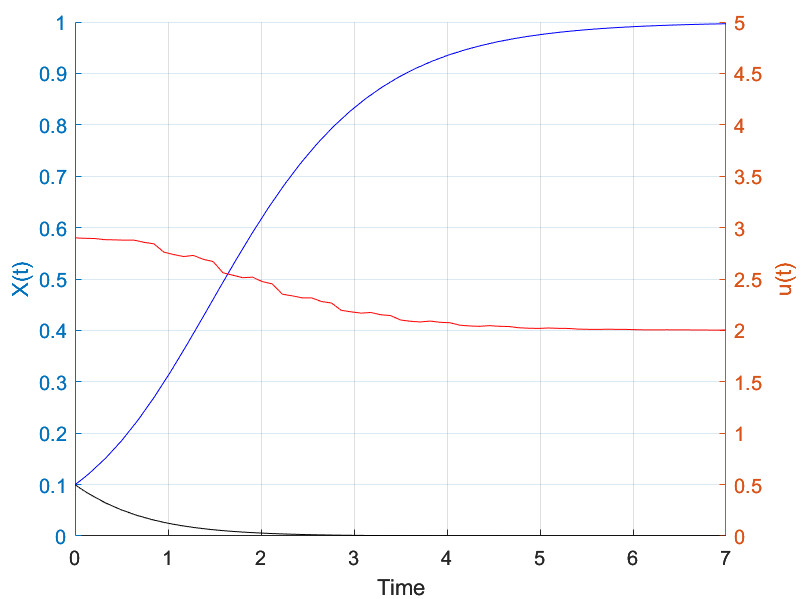}
	\caption[Simulation of Optimal Incentive in Well-mixed System]{Optimal control $=-2\big(kx+l(1-x)\big)$ (equation \ref{uscalar}) for a well-mixed population, illustrating the results of Lemma \ref{scalarequilibrium} and Theorem \ref{optscalar}. The payoff matrix is 
    $\mathcal{P}=\begin{bmatrix}
2 & 0\\
3 & 1.5\\
\end{bmatrix}$. The left vertical axis represents the cooperation level $x$ for both the uncontrolled Prisoner's Dilemma dynamics, which converge to defection, and the incentive-controlled dynamics, which converge to cooperation. The right vertical axis represents the incentive level corresponding to the black curve.}
	\label{fig:wellmix}
\end{figure}

The figure \ref{fig:wellmix} illustrates the evolution of cooperation in the Prisoner's Dilemma under the scalar dynamical system. It compares the uncontrolled dynamics, which converge to defection, with the incentivised dynamics, which converge to full cooperation. These results support Lemma \ref{scalarequilibrium} and Theorem \ref{optscalar}. The figure also shows the evolution of the incentive $u$ over time. As predicted by the analytical solution, the incentive adapts dynamically to the current fraction of cooperators in the population.

\begin{figure}
	\centering
	\includegraphics[width=1\textwidth]{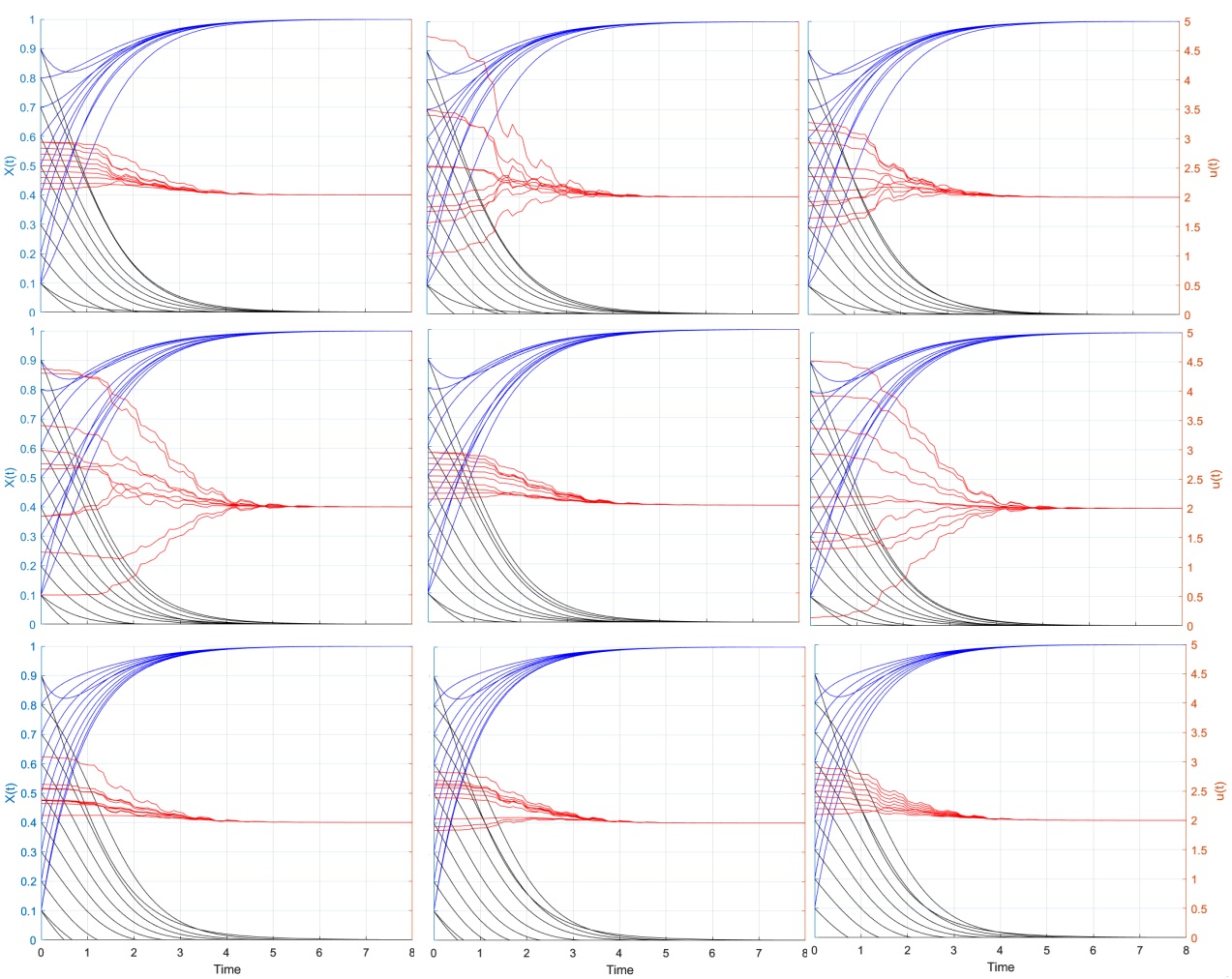}
	\caption[Simulation for Optimal Incentive in System with Single-layer and Multi-layer Network]{Optimal control for single-layer and multilayer network configurations. The payoff matrix is $\mathcal{P}=\begin{bmatrix}
2 & 0\\
3 & 1.5\\
\end{bmatrix}$. The left vertical axis represents the cooperation intensities of the ten nodes, with initial condition $\mathbf{x}=\left[0.1~0.1~0.2~0.3~0.4~0.5~0.6~0.7~0.8~0.9\right]$, for both the uncontrolled Prisoner's Dilemma dynamics that converge to defection (black lines) and the incentive-controlled dynamics that converge to cooperation (blue lines), consistent with Lemmas \ref{vectorequilibrium} and \ref{multilayerequilibrium}. The right vertical axis represents the corresponding incentive levels. The right vertical axis represents the incentive level corresponding to the red curves. For the single-layer system, the incentive is given by $\mathbf{u}^*=-2G_v^\top \Big(G_vG_v^\top\Big)^{-1}\mathbf{f}_v$ (equation \ref{u2veccandidate}, Theorem \ref{optinvertible}), whereas for the multilayer system it is given by $\mathbf{u}^*=-2G_m^\top \Big(G_mG_m^\top\Big)^{-1}\mathbf{f}_m$ (equation \ref{u2veccandidateMul}, Theorem \ref{optmultiinver}). Each row and column corresponds to a different network structure in the following order: Erdős-Rényi, Small-World, and Star. Consequently, the diagonal panels represent single-layer networks ($A_1^*=A_2^*$) with Erdős-Rényi, Small-World, and Star topologies, respectively, whereas the off-diagonal panels represent multilayer network combinations}
	\label{fig:multilayer}
\end{figure}

Figure \ref{fig:multilayer} compares the implementation of the optimal incentive in both single-layer and multilayer networked systems. The $3\times3$ panel presents different combinations of network structures. Similar to Figure \ref{fig:wellmix}, each subplot shows the system trajectories without incentive, the trajectories with incentive, and the corresponding incentive levels, which vary dynamically over time and across nodes. These results validate Lemma \ref{vectorequilibrium} and Theorem \ref{optinvertible} for the single-layer system, and Lemma \ref{multilayerequilibrium} and Theorem \ref{optmultiinver} for the multilayer system. Consistent with the well-mixed case, the Prisoner's Dilemma drives the population toward defection in the absence of incentives, whereas the proposed incentive mechanism promotes cooperation. The results also show that incentives are allocated optimally across the network, with node-specific values determined by the network topology. Furthermore, for a fixed topology, different initial conditions lead to different incentive profiles and system trajectories over time.

The first row and column of Figure \ref{fig:multilayer} correspond to the Erdős-Rényi network, the second row and column to the Small-World network, and the third row and column to the Star network. Consequently, the diagonal subplots represent single-layer network models, where the game interaction and information diffusion layers share the same network structure. From the upper-left to the lower-right corner, these correspond to the Erdős-Rényi, Small-World, and Star networks, respectively.

The off-diagonal subplots in Figure \ref{fig:multilayer} represent multilayer network configurations, as described in the previous paragraph. For example, the first row uses the Erdős-Rényi network as the first layer, while the second layer is either a Small-World or a Star network. The results show that the network structure of both layers influences the incentive allocation over time. As established in Theorem \ref{optmultiinver}, the optimal incentive strategy promotes cooperation throughout the population. The incentive calculation depends on the initial conditions, the diffusion-layer network $A_1^*$, the game-layer network $A^*_2$, and the payoff matrix. Figure \ref{fig:multilayer} further shows that, under similar initial conditions and payoff matrices, different multilayer network configurations produce different incentive trajectories while still achieving full cooperation.

\begin{figure}
    \centering
    \includegraphics[width=0.8\linewidth]{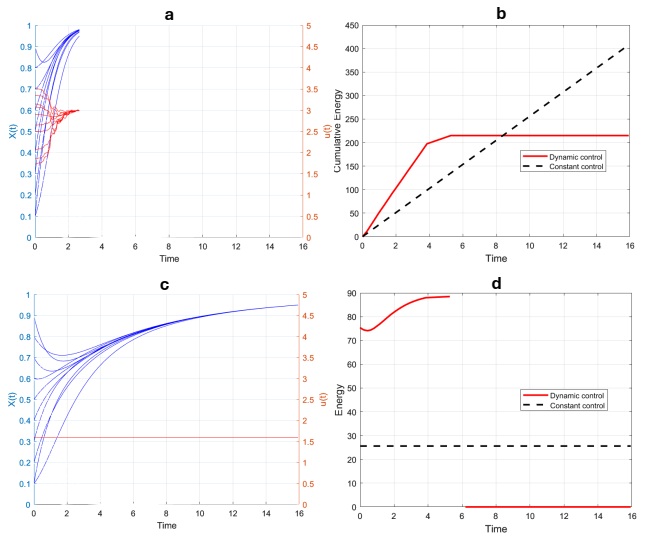}
    \caption[Total Energy Comparison of Optimal Dynamic and Constant Incentive]{Comparison of the cooperation trajectories and incentive level obtained using the optimal control $\mathbf{u}^*=-2G_m^\top \Big(G_mG_m^\top\Big)^{-1}\mathbf{f}_m$ (equation \ref{u2veccandidateMul}, Theorem \ref{optmultiinver}) (panel a) and a constant control input (panel c). Panels (b) and (d) compare the all nodes cost for each time and the cumulative, respectively. In panels (a) and (c), the left vertical axis represents the cooperation intensity $x_i$ (blue lines), while the right vertical axis represents the incentive level $u_i$ (red lines). In panel (b), the vertical axis represents the cumulative incentive energy $J=\int_{0}^{t_f}\frac{1}{2}\,\mathbf{u}^{\top}\mathbf{u}\,dt$, whereas in panel (d) it represents the instantaneous incentive energy $\mathbf{u}(t)^{\top}\mathbf{u}(t)$ for all nodes. Both quantities are shown for the optimal and constant control strategies. The payoff matrix is $\mathcal{P}=\begin{bmatrix}
2 & 0\\
3 & 1.5\\
\end{bmatrix}$ which implies $k=-1$ and $l=-1.5$. The constant control input of $1.6$ is selected because it is sufficiently small while still making the payoff difference positive, thereby transforming the Prisoner's Dilemma into a Harmony Game. The first layer is an Erdős-Rényi network, and the second layer is a Star network (Figure \ref{fig:graphplot}(b)) The simulation will stop when $x_i \geq 0.95$ for all nodes $i$.}
    \label{fig:energycompare}
\end{figure}

Figure \ref{fig:energycompare} compares the cumulative incentive energy obtained using the optimal control law $\mathbf{u}^*=-2G_m^\top \Big(G_mG_m^\top\Big)^{-1}\mathbf{f}_m$ from equation \ref{u2veccandidateMul} and Theorem \ref{optmultiinver} with that obtained using a constant incentive applied to all nodes throughout the simulation. The results show that, over time, the proposed optimal control strategy is less energy expensive than the constant control strategy, thereby supporting the result established in Theorem \ref{optmultiinver}.

\begin{figure}
	\centering
	\includegraphics[width=1\textwidth]{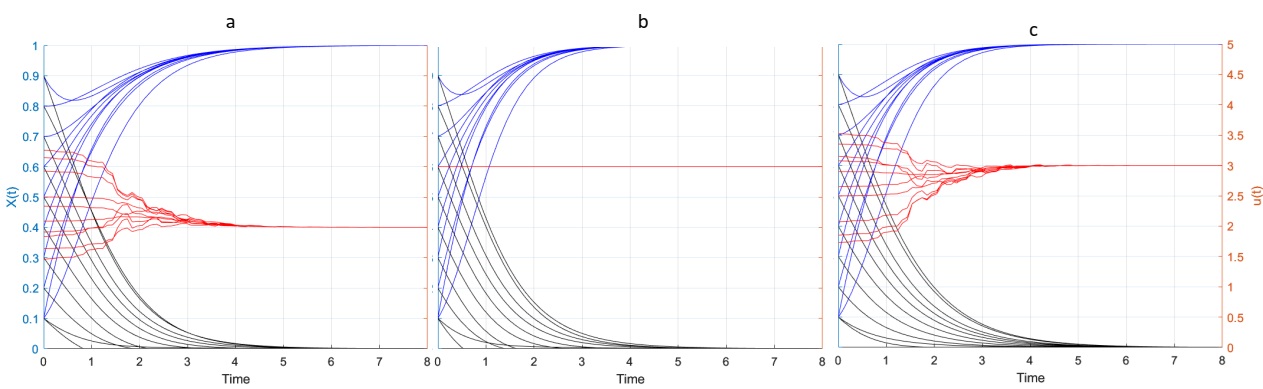}
	\caption[Optimal Incentive Comparison Between Different Payoff Matrix]{Optimal incentive $\mathbf{u}^*=-2G_v^\top \Big(G_vG_v^\top\Big)^{-1}\mathbf{f}_v$ (equation \ref{u2veccandidate}, Theorem \ref{optinvertible}) for different values of $k$ and $l$ in a single-layer network. From left to right, the payoff matrices are $\mathcal{P}=\begin{bmatrix}
2 & 0\\
3 & 1.5\\
\end{bmatrix}$, $\begin{bmatrix}
2 & 0\\
3.5 & 1.5\\
\end{bmatrix}$, and $\begin{bmatrix}
2 & 0\\
3.5 & 1\\
\end{bmatrix}$, corresponding to the cases $k>l$, $k=l$, and $k<l$, respectively. The initial conditions, axes, and legends are the same as those in Figure \ref{fig:multilayer}.}
	\label{fig:comparegame}
\end{figure}

Figure \ref{fig:comparegame} presents simulations on single-layer networks under different payoff matrices. In the Prisoner's Dilemma, both $k<0$ and $l<0$. The results show that when $k>l$ (left panel), the incentive decreases over time as the system approaches full cooperation. In contrast, when $k<l$ (right panel), the incentive increases over time to drive the system toward the target terminal state. When $k=l$ (middle panel), the incentive remains constant over time, consistent with the result reported in \cite{wang2025optimally}.

\section{Discussion and Conclusions}
This work addresses a social dilemma in which each agent is individually incentivised to defect for personal benefit rather than collective welfare. As a result, the population may converge to widespread defection, despite mutual cooperation yielding a higher collective payoff. The Prisoner's Dilemma provides a suitable framework for modelling this behaviour. To capture player interactions, we consider both single-layer and multilayer network structures, recognising that in real-world systems, individuals often participate in multiple communities with distinct interaction structures. In addition, the proposed model adopts continuous strategies rather than binary cooperate-or-defect decisions and is formulated on general rather than regular network topologies. This framework enables the design of decentralised, time-varying, and node-specific optimal incentives that depend on each player’s current cooperation level and structural position within the multilayer network.

We first present Lemma \ref{scalarequilibrium}, which establishes the equilibrium behaviour of the replicator dynamics under the Prisoner's Dilemma and Harmony game settings in a well-mixed population. Subsequently, Theorem \ref{optscalar} derives the candidate optimal control strategy for the case of free terminal time and fixed terminal state. The result shows that the incentive is dynamic and depends on feedback from the current strategy distribution within the population.

Based on the results of the well-mixed system, we next consider its implementation in a single-layer networked system. Following the same analytical approach, Lemma \ref{vectorequilibrium} establishes the stable boundary equilibria of the Prisoner's Dilemma and Harmony Game using Lyapunov stability theory. Theorem \ref{optinvertible} derives the candidate optimal control that minimises the incentive energy while satisfying the terminal state constraint. The results show that the incentive not only varies dynamically over time but also differs across nodes as a function of each player’s current cooperation level and the network topology. This extends previous findings on optimal control in regular lattice networks, where each node has similar neighbourhood structures and incentives are distributed more uniformly across nodes \cite{wang2022decentralized}. In contrast, the use of a general network and mixed strategies in our framework leads to heterogeneous neighbourhood structures and node-specific incentives.

The single-layer network framework also enables extension to multilayer networks. We further consider the implementation of optimal incentives in a multilayer setting, where game interactions and strategy diffusion occur over distinct network structures. Consistently, the stable all-cooperation and all-defection equilibria for the Prisoner's Dilemma and Harmony Game in the multilayer network are established in Lemma \ref{multilayerequilibrium}. Subsequently, Theorem \ref{optmultiinver} proves that the proposed optimal control strategy, derived using the Hamilton-Jacobi-Bellman method, yields an optimal incentive mechanism for promoting the spread of cooperation in the Prisoner's Dilemma.

The proposed optimal control strategy is consistent with previous analytical findings showing that full cooperation can be achieved under free terminal time and free terminal state settings \cite{wang2022decentralized}\cite{wang2023optimization}. Our simulations are also consistent with previous results showing that the incentive may remain constant when the payoff matrix is structured such that $k=l$ \cite{wang2025optimally}, meaning that the relative payoff difference between cooperation and defection is independent of the opponent’s strategy. However, this work extends the previous result by considering general network structures, continuous cooperation as dynamic variables rather than discrete strategies, and multilayer network interactions. While previous work considers reward and punishment as mutually exclusive control actions \cite{wang2022decentralized}, our framework allows both mechanisms to operate simultaneously across different nodes. As a result, some nodes may receive rewards, while others may receive penalties at the same time. The proposed multilayer optimal control framework also offers potential applications in various domains, such as socio-technical systems, where social and technical subsystems may have distinct but interconnected interaction structures \cite{Ropohl1999}\cite{zhu2025revisiting}.

This work also opens several directions for future research. The proposed framework and optimal control solution may be extended to stopping and switching systems and noisy dynamics, as the solution is formulated as a state-feedback solution within the Hamilton-Jacobi-Bellman framework, in which the optimal decision is determined by the current state, regardless of the preceding trajectory, including cases involving disruptions or disconnections from previous time steps. Future research may also explore the application of the proposed framework to other game settings, such as the Snowdrift and Stag Hunt Games.

\section*{Acknowledgement}
This work was supported by the Beasiswa Pendidikan Indonesia (BPI) Scholarship, PPAPT Kemdiktisaintek, and LPDP under the Government of Indonesia.




\bibliographystyle{unsrtnat}

\bibliography{citation}



\end{document}